\documentclass[11pt]{amsart}

\usepackage{amsmath,amssymb,amsthm,mathtools}
\usepackage{mathrsfs}
\usepackage[margin=1.12in]{geometry}
\usepackage{microtype}
\usepackage[colorlinks=true,linkcolor=blue,citecolor=blue,urlcolor=blue]{hyperref}
\hypersetup{
 unicode=true,
 pdftitle={Universal Plücker positivity and the Octopus inequality},
 pdfauthor={Haoran Zhu},
 pdfsubject={Real Schubert calculus, symmetric groups, and operator inequalities},
 pdfkeywords={Octopus inequality, universal Plücker coordinates, symmetric group,
 Young branching, positive semidefinite elements}
}

\newtheorem{theorem}{Theorem}[section]
\newtheorem{proposition}[theorem]{Proposition}
\newtheorem{lemma}[theorem]{Lemma}
\newtheorem{corollary}[theorem]{Corollary}
\theoremstyle{definition}

\theoremstyle{remark}
\newtheorem{remark}[theorem]{Remark}

\numberwithin{equation}{section}

\newcommand{\C}{\mathbb C}
\newcommand{\Sn}{\mathfrak S}
\newcommand{\one}{1}
\newcommand{\abs}[1]{\left\lvert #1\right\rvert}
\newcommand{\set}[1]{\left\{#1\right\}}
\newcommand{\Ind}{\operatorname{Ind}}
\newcommand{\Res}{\operatorname{Res}}
\newcommand{\tr}{\operatorname{tr}}
\newcommand{\pcoord}{\mathscr P}
\newcommand{\trans}{\mathcal T}
\newcommand{\proj}{\mathsf P}
\newcommand{\f}{f}

\title[Universal Pl\"ucker positivity and the Octopus inequality]
{Universal Pl\"ucker positivity\\and the Octopus inequality}

\author{Haoran Zhu}
\address{Division of Mathematical Sciences, Nanyang Technological University,
Singapore 637371}
\email{zhuh0031@e.ntu.edu.sg}

\date{27 August 2026}
\subjclass[2020]{Primary 20C30; Secondary 05E10, 14M15, 15B48}
\keywords{Octopus inequality, universal Pl\"ucker coordinates, real Schubert
calculus, symmetric group, Young's branching rule, positive semidefinite elements}

\begin{document}

\begin{abstract}
We prove a positivity theorem for the universal Pl\"ucker coordinates of
Karp and Purbhoo when exactly one parameter is negative, with a sharp uniform
threshold governed by the largest Plancherel up-transition probability.  At
the critical specialisation, positivity of the normalised $(2,2)$-coordinate
is equivalent to the Octopus inequality, the main technical tool in
the proof of Aldous's spectral gap conjecture.  By connecting the inequality
with Schubert calculus, this answers questions raised by Caputo and
Aldous.  For a Wronskian with distinct real zeros, we
determine the
maximal common half-line on which all branching-balanced coordinates are
positive semidefinite.  We also give Pl\"ucker-theoretic proofs of two
hypergraph inequalities of Alon, Kozma, and Puder.
\end{abstract}

\maketitle

\section{Introduction}

Universal Pl\"ucker positivity and the Octopus inequality arise in apparently
unrelated areas of mathematics; we establish a direct connection between them.

Karp and Purbhoo~\cite{KarpPurbhoo} introduced \textbf{universal Pl\"ucker coordinates} in their
solution of the inverse Wronski problem~\cite[Theorem~1.3(v), Proposition~1.12, and
Section~1.3]{KarpPurbhoo}.  The normalised Pl\"ucker coordinates
of points in a Wronski fibre occur as eigenvalues of their commuting
symmetric-group operators, whose positivity establishes several conjectures
in real Schubert calculus.

The \textbf{Octopus inequality} is the most technical tool in the proof of
Aldous's spectral gap conjecture by Caputo, Liggett, and
Richthammer~\cite{CLR}.  The conjecture asserts that the interchange process
and the random walk on every finite weighted graph have the same spectral
gap.  
Subsequent work has developed applications and extensions to exclusion processes, operator inequalities, and spectral
gaps on hypergraphs~\cite{AlonKozma,AlonKozmaPuder,Cesi,Chen}.  Despite these
developments, the available proofs gave little indication of why the
inequality should hold.  Caputo
~\cite[\emph{Perspective}]{AldousCaputoDurrettHolroydJungPuha2021}
and Aldous~\cite[\emph{Postscript}]{AldousCaputoDurrettHolroydJungPuha2021}
sought a more illuminating explanation, possibly through a connection with
another area of mathematics.

Interstingly, at the boundary $\sum_a u_a^{-1}=0$, the first Pl\"ucker relation gives the
connection: it factors the normalised $(2,2)$-coordinate as the product
of the left-hand side of the Octopus inequality and its image under the sign
automorphism.  The spectral theorem, together with the sign automorphism,
shows that positivity of the coordinate is equivalent to the Octopus
inequality.  Thus the latter is precisely the critical $(2,2)$ case of
universal Pl\"ucker positivity.  The same coordinate also gives
Pl\"ucker-theoretic proofs of two hypergraph generalisations~\cite{AlonKozmaPuder}.

Let $\one$ denote the identity of $\C[\Sn_n]$, and let $*$ be the antilinear
involution determined by $\sigma^*=\sigma^{-1}$.  For self-adjoint elements
$A,B\in\C[\Sn_n]$, write $A\succeq0$ if $\rho(A)$ is positive semidefinite
in every unitary representation $\rho$ of $\Sn_n$, and write $A\preceq B$
if $B-A\succeq0$.

For $\mathbf u=(u_1,\ldots,u_n)$, write
$\pcoord^\lambda(\mathbf u)\in\C[\Sn_n]$ for the universal coordinate (see a explict definition in Section~\ref{sec:universal}) indexed
by $\lambda$.  We determine a sharp uniform threshold for its positivity
when one parameter is negative.  The threshold is determined by the largest
Plancherel up-transition probability.

Let $\lambda\vdash k$, and write $\f^\lambda=\dim S^\lambda$.  If $\mu$ is
obtained from $\lambda$ by adding one box, write $\mu\gtrdot\lambda$.  The
Plancherel up-transition probability is given
by~\cite[Section~2.2]{Hora}:
\begin{equation}\label{eq:up-prob-intro}
 p^\uparrow(\lambda,\mu)
 =\frac{\f^\mu}{(k+1)\f^\lambda}.
\end{equation}
For a nonempty partition $\lambda$, write
\[
 \vartheta(\lambda)=\max_{\mu\gtrdot\lambda}
 p^\uparrow(\lambda,\mu).
\]
Every nonempty partition has at least two covers, so
$0<\vartheta(\lambda)<1$.  We say that $\lambda$ is
\textbf{branching-balanced} if
\[
 \vartheta(\lambda)\leqslant\frac12.
\]

\begin{theorem}\label{thm:one-negative}
Let $1\leqslant k\leqslant n$, let $\lambda\vdash k$, and let
$(x_i)_{1\leqslant i<n}$ be a family of positive real numbers in
non-increasing order.  If $a>0$ and
\[
 a\geqslant
 \frac{\vartheta(\lambda)}{1-\vartheta(\lambda)}
 \sum_{i=1}^{n-k}x_i,
\]
where the sum is zero when $k=n$, then
\[
 \pcoord^\lambda
 \bigl(x_1^{-1},\ldots,x_{n-1}^{-1},-a^{-1}\bigr)\succeq0.
\]
The coefficient $\vartheta(\lambda)/(1-\vartheta(\lambda))$ is sharp
uniformly in $n$ and the weights, already when $n=k+1$ and
$x_1=\cdots=x_k$.
\end{theorem}

In particular, if $\lambda$ is branching-balanced and
$u_1,\ldots,u_{n-1}>0>u_n$, then
$\pcoord^\lambda(\mathbf u)\succeq0$ whenever
$\sum_{i=1}^n u_i^{-1}\leqslant0$.
This result also has a consequence for Schubert calculus.  Let
$g(t)=\prod_a(t-r_a)$, where $r_1<\cdots<r_n$, and let $\xi$ be the largest
zero of $g'$.  We prove that
$\pcoord^\lambda(t-r_1,\ldots,t-r_n)\succeq0$ for every $t\geqslant\xi$
and every nonempty branching-balanced $\lambda$ with $\abs\lambda\leqslant n$.
The left endpoint is sharp uniformly in $\lambda$.

At the critical specialisation, positivity of the $(2,2)$-coordinate is
equivalent to the following inequality.

\begin{theorem}[Octopus inequality]\label{thm:octopus}
Let $n\geqslant2$ and let $x_1,\ldots,x_{n-1}$ be nonnegative real
numbers.  Then
\[
 \left(\sum_{i<n}x_i\right)
 \sum_{i<n}x_i\bigl(\one-(i\,n)\bigr)
 -\sum_{i<j<n}x_ix_j\bigl(\one-(i\,j)\bigr)\succeq0.
\]
\end{theorem}

The same coordinate gives Pl\"ucker-theoretic proofs of two hypergraph
inequalities of Alon, Kozma, and Puder
\cite[Theorems~1.7 and~1.9]{AlonKozmaPuder}.  Their remaining generalised
Octopus inequality, for sets with large
intersection~\cite[Theorem~1.8]{AlonKozmaPuder}, is discussed in the final
remark.

\section{Branching projections}\label{sec:branching}

For every finite group $G$, we use the analogous involution on $\C[G]$,
given by the antilinear extension of $g^*=g^{-1}$, and the same order on
self-adjoint elements.  It is enough to look at the left regular representation.
If $H\leqslant G$, then the standard inclusion
$\C[H]\hookrightarrow\C[G]$ preserves positivity.  The order is also closed
under nonnegative linear combinations.

Let $X$ be a finite set of cardinality $k$, and let $\Sn_X$ be the symmetric
group on $X$.  For $\lambda\vdash k$, let $\chi^\lambda$ be the irreducible
character of $S^\lambda$ and let $\f^\lambda=\dim S^\lambda$.  The standard
character formula~\cite[Section~3.4]{FultonHarris} gives the central primitive
idempotent corresponding to $S^\lambda$:
\begin{equation}\label{eq:central-idempotent}
 \proj_X^\lambda
 =\frac{\f^\lambda}{k!}
  \sum_{\sigma\in\Sn_X}\chi^\lambda(\sigma^{-1})\sigma
 \in\C[\Sn_X].
\end{equation}
Characters of symmetric groups are real and constant on cycle type, so
$\chi^\lambda(\sigma^{-1})=\chi^\lambda(\sigma)$.  In any unitary
representation of a symmetric group containing $\Sn_X$, the element
$\proj_X^\lambda$ acts as the orthogonal projection onto the
$S^\lambda$-isotypic component of the restriction to $\Sn_X$.

We use the branching rule~\cite[Section~2.8]{Sagan} in the forms

\noindent
\begin{minipage}[t]{0.49\textwidth}
\begin{equation}\label{eq:ind-branching}
 \Ind_{\Sn_k}^{\Sn_{k+1}}S^\lambda
 \cong\bigoplus_{\mu\gtrdot\lambda}S^\mu.
\end{equation}
\end{minipage}\hfill
\begin{minipage}[t]{0.49\textwidth}
\begin{equation}\label{eq:res-branching}
 \Res_{\Sn_k}^{\Sn_{k+1}}S^\mu
 \cong\bigoplus_{\lambda\lessdot\mu}S^\lambda.
\end{equation}
\end{minipage}

Both decompositions are multiplicity-free.  Taking dimensions in
\eqref{eq:ind-branching} gives
\begin{equation}\label{eq:dimension-branching}
 \sum_{\mu\gtrdot\lambda}\f^\mu
 =(k+1)\f^\lambda.
\end{equation}
This verifies that \eqref{eq:up-prob-intro} is a probability distribution.

We first need the following branching identity.

\begin{proposition}\label{prop:branching-projection}
Let $Y$ be a set of cardinality $k+1$, and let $\lambda\vdash k$.  Then, in
$\C[\Sn_Y]$,
\[
 \sum_{q\in Y}\proj_{Y\setminus\set{q}}^{\lambda}
 =\sum_{\mu\gtrdot\lambda}
   \frac{(k+1)\f^\lambda}{\f^\mu}\,\proj_Y^\mu
 =\sum_{\mu\gtrdot\lambda}
   p^\uparrow(\lambda,\mu)^{-1}\proj_Y^\mu.
\]
\end{proposition}

\begin{proof}
Conjugation by $g\in\Sn_Y$ sends the summand indexed by $q$ to the summand
indexed by $g(q)$.  Hence the left-hand side is central in $\C[\Sn_Y]$.

Let $\mu\vdash k+1$.  By \eqref{eq:res-branching}, the trace of
$\proj_{Y\setminus\set{q}}^{\lambda}$ on $S^\mu$ is $\f^\lambda$ if
$\mu\gtrdot\lambda$, and is zero otherwise.  The trace is independent of
$q$, because the subgroups $\Sn_{Y\setminus\set{q}}$ are conjugate.  Thus
\[
 \sum_{q\in Y}\tr_{S^\mu}
   \bigl(\proj_{Y\setminus\set{q}}^{\lambda}\bigr)
 =\begin{cases}
   (k+1)\f^\lambda,&\mu\gtrdot\lambda,\\
   0,&\text{otherwise}.
  \end{cases}
\]
Schur's lemma~\cite[Theorem~1.6.5]{Sagan} now shows that the left-hand side
acts on $S^\mu$ by a scalar.  Dividing the displayed trace by $\f^\mu$
gives the first equality; the second follows from
\eqref{eq:up-prob-intro}.
\end{proof}

The next result gives the complete local spectrum.

\begin{proposition}\label{prop:local-criterion}
Let $\lambda\vdash k$ be nonempty, let $Y=A\sqcup\set{y}$, where
$\abs A=k$, and let $c\geqslant0$.  If $\mu\vdash k+1$ and
$\mu\gtrdot\lambda$, then, on $S^\mu$, the operator
\[
 \sum_{r\in A}
 \proj_{(A\setminus\set{r})\cup\set{y}}^\lambda
 -c\proj_A^\lambda
\]
has eigenvalues $p^\uparrow(\lambda,\mu)^{-1}-(1+c)$ and
$p^\uparrow(\lambda,\mu)^{-1}$, with multiplicities $\f^\lambda$ and
$\f^\mu-\f^\lambda$, respectively.  If $\mu\vdash k+1$ and
$\mu\not\gtrdot\lambda$, it vanishes on $S^\mu$.  Consequently, it is positive
semidefinite if and only if
$\vartheta(\lambda)\leqslant(1+c)^{-1}$.
\end{proposition}

\begin{proof}
By Proposition~\ref{prop:branching-projection}, the operator is
\[
 \sum_{\mu\gtrdot\lambda}
 \left[
 p^\uparrow(\lambda,\mu)^{-1}
   \bigl(\proj_Y^\mu-\proj_Y^\mu\proj_A^\lambda\bigr)
 +\bigl(p^\uparrow(\lambda,\mu)^{-1}-(1+c)\bigr)
   \proj_Y^\mu\proj_A^\lambda
 \right].
\]
Since $\proj_Y^\mu$ is central, it commutes with $\proj_A^\lambda$, so their
product is an orthogonal projection.  By the multiplicity-free branching
rule, this projection has rank $\f^\lambda$ on $S^\mu$ when
$\mu\gtrdot\lambda$, and vanishes otherwise.  The assertions follow.
\end{proof}

We now apply the criterion to square partitions.

\begin{corollary}\label{cor:square-balanced}
Let $r\geqslant1$.  Then the square partition $(r^r)$ is
branching-balanced.
\end{corollary}

\begin{proof}
The two covers $(r+1,r^{r-1})$ and $(r^r,1)$ are conjugate; here
$r^{r-1}$ is omitted when $r=1$.  Their Specht modules therefore have the
same dimension.  By \eqref{eq:dimension-branching},
\[
 \f^{(r+1,r^{r-1})}=\f^{(r^r,1)}
 =\frac{r^2+1}{2}\f^{(r^r)},
\]
and both up-transition probabilities are $1/2$.
\end{proof}


\section{Universal Pl\"ucker positivity}\label{sec:universal}

Let $[n]=\set{1,\ldots,n}$.  If $X\subseteq[n]$, we regard $\Sn_X$ as the
subgroup of $\Sn_n$ fixing the complement of $X$.  For
$\mathbf u=(u_1,\ldots,u_n)\in\C^n$ and $\lambda\vdash k$, the universal
Pl\"ucker coordinate is
\begin{equation}\label{eq:universal-coordinate}
 \pcoord^\lambda(\mathbf u)
 =\sum_{\substack{X\subseteq[n]\\\abs X=k}}
   \left(\sum_{\sigma\in\Sn_X}\chi^\lambda(\sigma)\sigma\right)
   \prod_{a\in[n]\setminus X}u_a
 \in\C[\Sn_n].
\end{equation}
We put $\pcoord^\lambda(\mathbf u)=0$ when $k>n$.  For the empty
partition,
\begin{equation}\label{eq:P-empty}
 \pcoord^\varnothing(\mathbf u)=(u_1\cdots u_n)\one.
\end{equation}
These are the coordinates of Karp and Purbhoo with $t=0$ in their
notation~\cite[Equation~(1.2)]{KarpPurbhoo}.  Their universal Pl\"ucker
theorem shows that the elements $\pcoord^\lambda(\mathbf u)$ commute and
satisfy the quadratic Pl\"ucker relations
\cite[Theorem~1.3(i) and (iii)]{KarpPurbhoo}.

Let $x_1,\ldots,x_{n-1},a>0$, and put
\[
 u_i=x_i^{-1}\quad(i<n),
 \qquad u_n=-a^{-1}.
\]
Using \eqref{eq:central-idempotent} in
\eqref{eq:universal-coordinate}, and separating the subsets which contain
$n$, gives
\begin{equation}\label{eq:signed-projections}
\begin{aligned}
 -\frac{1}{u_1\cdots u_n}\pcoord^\lambda(\mathbf u)
 =\frac{k!}{\f^\lambda}\Biggl(&
 a\sum_{\substack{B\subseteq[n-1]\\\abs B=k-1}}
    \left(\prod_{i\in B}x_i\right)
    \proj_{B\cup\set{n}}^\lambda-\sum_{\substack{A\subseteq[n-1]\\\abs A=k}}
    \left(\prod_{i\in A}x_i\right)\proj_A^\lambda
 \Biggr).
\end{aligned}
\end{equation}
An empty sum is understood to be zero.

\begin{proof}[Proof of Theorem~\ref{thm:one-negative}]
The expression in parentheses in \eqref{eq:signed-projections} is equal to
\[
\begin{aligned}
 &\frac{\vartheta(\lambda)}{1-\vartheta(\lambda)}
 \sum_{\substack{A\subseteq[n-1]\\\abs A=k}}
 \left(\prod_{i\in A}x_i\right)
 \left(
  \sum_{r\in A}
   \proj_{(A\setminus\set{r})\cup\set{n}}^\lambda
  -\frac{1-\vartheta(\lambda)}{\vartheta(\lambda)}
   \proj_A^\lambda
 \right)\\
 &\quad+
 \sum_{\substack{B\subseteq[n-1]\\\abs B=k-1}}
 \left(\prod_{i\in B}x_i\right)
 \left(
  a-\frac{\vartheta(\lambda)}{1-\vartheta(\lambda)}
    \sum_{r\in[n-1]\setminus B}x_r
 \right)\proj_{B\cup\set{n}}^\lambda.
\end{aligned}
\]
Every term in the first sum is positive semidefinite by
Proposition~\ref{prop:local-criterion}.  The complement of $B$ has $n-k$
elements, and hence
\[
 \sum_{r\in[n-1]\setminus B}x_r
 \leqslant\sum_{i=1}^{n-k}x_i.
\]
The coefficients in the second sum are therefore nonnegative.  Since
$u_1\cdots u_n<0$, \eqref{eq:signed-projections} proves the positivity.

For sharpness, take $n=k+1$ and $x_1=\cdots=x_k=x$.  Then the right-hand
side of \eqref{eq:signed-projections} is
\[
 \frac{k!}{\f^\lambda}x^k
 \left(
  \frac{a}{x}\sum_{r=1}^k
   \proj_{([k]\setminus\set{r})\cup\set{k+1}}^\lambda
  -\proj_{[k]}^\lambda
 \right).
\]
Proposition~\ref{prop:local-criterion}, applied with $c=x/a$, shows that
this is positive semidefinite if and only if
\[
 a\geqslant
 \frac{\vartheta(\lambda)}{1-\vartheta(\lambda)}x.\qedhere
\]
\end{proof}

The balanced case takes a particularly simple form.

\begin{corollary}\label{cor:balanced-one-negative}
Let $1\leqslant k\leqslant n$, let $\lambda\vdash k$ be branching-balanced,
and let $u_1,\ldots,u_{n-1}>0>u_n$.  If
\[
 \sum_{i=1}^n u_i^{-1}\leqslant0,
\]
then $\pcoord^\lambda(\mathbf u)\succeq0$.
\end{corollary}

\begin{proof}
Simultaneously relabelling the positive parameters and the letters of
$\Sn_n$ conjugates the coordinate and preserves positivity.  We may
therefore apply Theorem~\ref{thm:one-negative} with $x_i=u_i^{-1}$ and
$a=-u_n^{-1}$.  The hypothesis gives $a\geqslant\sum_{i<n}x_i$, while
$\vartheta(\lambda)/(1-\vartheta(\lambda))\leqslant1$.  Hence
\[
 a\geqslant\sum_{i=1}^{n-k}x_i
 \geqslant\frac{\vartheta(\lambda)}{1-\vartheta(\lambda)}
          \sum_{i=1}^{n-k}x_i,
\]
and the theorem applies.
\end{proof}

We next apply Corollary~\ref{cor:balanced-one-negative} to a Wronskian with
distinct real zeros.

\begin{corollary}\label{cor:outer-critical}
Let $n\geqslant2$, let
\[
 g(t)=\prod_{a=1}^n(t-r_a),
 \qquad r_1<\cdots<r_n,
\]
and let $\xi\in(r_{n-1},r_n)$ be the largest zero of $g'$.  If
$1\leqslant k\leqslant n$ and $\lambda\vdash k$ is branching-balanced, then
\[
 \pcoord^\lambda(t-r_1,\ldots,t-r_n)\succeq0
 \qquad\text{for every }t\geqslant\xi.
\]
This is sharp uniformly in $\lambda$: if $\eta<\xi$, then the
assertion fails on $[\eta,\infty)$ for some nonempty branching-balanced
$\lambda$ with $\abs\lambda\leqslant n$.
\end{corollary}

\begin{proof}
Suppose first that $\xi\leqslant t<r_n$.  The first $n-1$ parameters are
positive and the last is negative.  Since $g(t)<0$ and $g'(t)\geqslant0$,
\[
 \sum_{a=1}^n\frac{1}{t-r_a}=\frac{g'(t)}{g(t)}\leqslant0.
\]
The assertion follows from Corollary~\ref{cor:balanced-one-negative}.  For
$t\geqslant r_n$, all parameters are nonnegative, and the assertion follows
from~\cite[Proposition~1.12]{KarpPurbhoo}.

The partition $(1)$ is branching-balanced and
\[
 \pcoord^{(1)}(t-r_1,\ldots,t-r_n)=g'(t)\one.
\]
Since $g'(t)<0$ on $(r_{n-1},\xi)$, every half-line with left endpoint below
$\xi$ contains a point where this coordinate is negative.
\end{proof}

For the parameter vector $(t-r_1,\ldots,t-r_n)$, the corresponding Wronskian
is the translate $v\mapsto g(v+t)$.  On each common eigenspace in a Specht
module, the eigenvalues of the universal coordinates are the normalised Pl\"ucker
coordinates of the corresponding point in the fibre of the Wronski
map~\cite[Theorem~1.3(v)]{KarpPurbhoo}.  Thus the positivity known for
$t\geqslant r_n$ extends to $t\geqslant\xi$.

\section{The critical coordinate and the Octopus inequality}
\label{sec:critical}

We shall use the following quadratic Pl\"ucker relation
\cite[Equation~(2.6)]{KarpPurbhoo}:
\begin{equation}\label{eq:first-Plucker}
 -\pcoord^\varnothing\pcoord^{(2,2)}
 +\pcoord^{(1)}\pcoord^{(2,1)}
 -\pcoord^{(1,1)}\pcoord^{(2)}=0.
\end{equation}
This is the Klein quadric relation, written in partition notation.

Let $n\geqslant2$ and let $x_1,\ldots,x_{n-1}>0$.  Put
\begin{equation}\label{eq:critical-intro}
 u_i=x_i^{-1}\quad(i<n),
 \qquad u_n=-\left(\sum_{i<n}x_i\right)^{-1}.
\end{equation}
Thus $u_1\cdots u_n<0$ and $\sum_a u_a^{-1}=0$.  If $n\geqslant4$, then
$(2,2)$ is branching-balanced by Corollary~\ref{cor:square-balanced}, and
Corollary~\ref{cor:balanced-one-negative} gives
\begin{equation}\label{eq:22-positive}
 -\frac{1}{u_1\cdots u_n}\pcoord^{(2,2)}(\mathbf u)\succeq0.
\end{equation}
When $n<4$, the coordinate vanishes by convention, so the assertion remains
valid.

Write
\begin{equation}\label{eq:omega-T}
 \omega=\sum_{i<n}x_i^2+\sum_{i<j<n}x_ix_j,
 \qquad
 \trans=\left(\sum_{i<n}x_i\right)
        \sum_{i<n}x_i(i\,n)-\sum_{i<j<n}x_ix_j(i\,j).
\end{equation}
Then
\begin{equation}\label{eq:critical-identities}
 \omega=\frac12\sum_{a=1}^nu_a^{-2},
 \qquad
 \sum_{a<b}\frac1{u_au_b}=-\omega,
 \qquad
 \sum_{a<b}\frac{(a\,b)}{u_au_b}=-\trans.
\end{equation}
Indeed, the first two identities follow by expanding
$\left(\sum_{i<n}x_i\right)^2$ and using
$\sum_a u_a^{-1}=0$, while the last follows directly from
\eqref{eq:critical-intro}.

\begin{lemma}\label{lem:low-coordinates}
Let $n\geqslant2$, and let $\mathbf u$, $\omega$, and $\trans$ be given by \eqref{eq:critical-intro} and \eqref{eq:omega-T}.  
Then
\[
\begin{aligned}
 \pcoord^\varnothing&=(u_1\cdots u_n)\one,
 &\qquad \pcoord^{(1)}&=0,\\
 \pcoord^{(1,1)}&=-(u_1\cdots u_n)(\omega\one-\trans),
 &\qquad \pcoord^{(2)}&=-(u_1\cdots u_n)(\omega\one+\trans).
\end{aligned}
\]
\end{lemma}

\begin{proof}
Equation \eqref{eq:P-empty} gives the first identity.  Since the character
of $\Sn_1$ is trivial,
\[
 \pcoord^{(1)}
 =(u_1\cdots u_n)\left(\sum_{a=1}^nu_a^{-1}\right)\one=0.
\]
The characters of $\Sn_2$ indexed by $(2)$ and $(1,1)$ are the trivial and
sign characters.  Therefore
\[
 \pcoord^{(1,1)}
 =(u_1\cdots u_n)\sum_{a<b}\frac{\one-(a\,b)}{u_au_b},
 \qquad
 \pcoord^{(2)}
 =(u_1\cdots u_n)\sum_{a<b}\frac{\one+(a\,b)}{u_au_b}.
\]
The remaining identities follow from \eqref{eq:critical-identities}.
\end{proof}

\begin{proposition}\label{prop:square-identity}
Let $n\geqslant2$, and let $\mathbf u$, $\omega$, and $\trans$ be given by \eqref{eq:critical-intro} and \eqref{eq:omega-T}.
Then
\[
 -\frac{1}{u_1\cdots u_n}\pcoord^{(2,2)}(\mathbf u)
 =\omega^2\one-\trans^2.
\]
\end{proposition}

\begin{proof}
Substitute Lemma~\ref{lem:low-coordinates} into
\eqref{eq:first-Plucker}.  The middle term vanishes and
\[
 (u_1\cdots u_n)\pcoord^{(2,2)}
 +(u_1\cdots u_n)^2(\omega\one-\trans)(\omega\one+\trans)=0.
\]
The two factors commute, being polynomials in $\trans$, and the result
follows on dividing by $(u_1\cdots u_n)^2$.
\end{proof}

On a common eigenspace for the universal coordinates, the vanishing of the
Pl\"ucker coordinate indexed by $(1)$ defines the tangent hyperplane to the
Klein quadric at the coordinate point indexed by $(2,1)$; the restriction of
the Klein relation to this hyperplane is the factorisation above.

\begin{corollary}\label{cor:symmetric-octopus}
Let $n\geqslant2$, let $x_1,\ldots,x_{n-1}>0$, and let $\mathbf u$, $\omega$,
and $\trans$ be given by \eqref{eq:critical-intro} and \eqref{eq:omega-T}.
Then
$\pcoord^{(2,2)}(\mathbf u)\succeq0$ if and only if
\[
 -\omega\one\preceq\trans\preceq\omega\one.
\]
Equivalently, every unitary representation $\rho$ of $\Sn_n$ satisfies
\[
 \lVert\rho(\trans)\rVert_{\mathrm{op}}\leqslant\omega.
\]
Both endpoints are attained.
\end{corollary}

\begin{proof}
Since $u_1\cdots u_n<0$, Proposition~\ref{prop:square-identity} shows that
positivity of the coordinate is equivalent to
$\omega^2\one-\trans^2\succeq0$.  Since $\trans$ is
self-adjoint, the spectral theorem gives the two-sided bound and the
operator-norm formulation.  In the trivial representation,
$\rho(\trans)=\omega$, while in the sign representation
$\rho(\trans)=-\omega$.

The upper bound is precisely the Octopus inequality.  Conversely, the sign
automorphism $\sigma\mapsto\operatorname{sgn}(\sigma)\sigma$ preserves
positivity and sends $\trans$ to $-\trans$, so the upper bound also gives
the lower bound.
\end{proof}

Now, surprisingly, the Octopus inequality appears.

\begin{proof}[Proof of Theorem~\ref{thm:octopus}]
For positive weights, the result follows from \eqref{eq:22-positive} and
Corollary~\ref{cor:symmetric-octopus}.  If some $x_i$ is zero, replace every
$x_i$ by $x_i+\varepsilon$ and let $\varepsilon\downarrow0$.  The positive
semidefinite cone is closed.
\end{proof}

\section{Hypergraph generalisations}\label{sec:hypergraph}

For a nonempty set $A\subseteq[n]$, write
\[
 \alpha_A=\abs A\bigl(\one-\proj_A^{(\abs A)}\bigr)
 =\frac{1}{(\abs A-1)!}\sum_{\sigma\in\Sn_A}(\one-\sigma),
\]
and put $\alpha_\varnothing=0$.  Thus $\alpha_A\succeq0$ and
$0\preceq\alpha_A\preceq\abs A\one$; in particular,
$\alpha_{\set{a,b}}=\one-(a\,b)$.  This is the notation of Alon, Kozma,
and Puder~\cite[Definition~1.5]{AlonKozmaPuder}.  We use the
$(2,2)$-positivity proved above to give Pl\"ucker-theoretic proofs of
two of their hypergraph inequalities
\cite[Theorems~1.7 and~1.9]{AlonKozmaPuder}.

We shall use the critical $(2,2)$-coordinate in the following form.  Fix
$q\in[n]$, let $\varnothing\neq X\subseteq[n]\setminus\set{q}$, and let
$y_a\geqslant0$ for $a\in X$.  For positive weights, take
$u_a=y_a^{-1}$ and $u_q=-(\sum_{a\in X}y_a)^{-1}$ in
$\C[\Sn_{X\cup\set{q}}]$.  After relabelling, Corollary~\ref{cor:symmetric-octopus}
and the standard inclusion into $\C[\Sn_n]$ give
\begin{equation}\label{eq:hypergraph-critical-22}
 \left(\sum_{a\in X}y_a\right)
 \sum_{a\in X}y_a\alpha_{\set{a,q}}
 \succeq
 \sum_{\set{a,b}\subseteq X}y_ay_b\alpha_{\set{a,b}}.
\end{equation}
Zero weights follow by continuity.

\begin{lemma}\label{lem:hypergraph-projections}
Let $A\subseteq[n]$ have cardinality $m\geqslant1$, and let
$q\in[n]\setminus A$.  Then, in $\C[\Sn_{A\cup\set{q}}]$,
\[
 \alpha_{A\cup\set{q}}\proj_A^{(m)}
 =\left(\sum_{a\in A}\alpha_{\set{a,q}}\right)\proj_A^{(m)}.
\]
Moreover, in $\C[\Sn_A]$,
\[
 \sum_{\set{a,b}\subseteq A}\alpha_{\set{a,b}}\succeq\alpha_A.
\]
\end{lemma}

\begin{proof}
The coset decomposition of $\Sn_{A\cup\set{q}}$ over $\Sn_A$ gives
\[
 (m+1)\proj_{A\cup\set{q}}^{(m+1)}
 =\left(\one+\sum_{a\in A}(a\,q)\right)\proj_A^{(m)}.
\]
Subtracting this equality from $(m+1)\proj_A^{(m)}$ gives the first identity.

For the second assertion, the two sides are central in $\C[\Sn_A]$.
They vanish on the trivial representation.  On $S^\lambda$, where
$\lambda\vdash m$ and $\lambda\neq(m)$, the right-hand side acts as
$m\one$, whereas the left-hand side acts as
\[
 \left(\binom{m}{2}
 -\sum_{(i,j)\in\lambda}(j-i)\right)\one.
\]
Here the sum of all transpositions acts by the content
sum~\cite[Section~4.1, Exercise~4.17(c)]{FultonHarris}.
If $\lambda\neq(m),(m-1,1)$, move the last box in the last nonempty row to
the end of the first row.  This gives another partition, and the content
sum increases by
\[
 \lambda_1-\lambda_j+j>0,
\]
where $j$ is the last nonempty row.  Repeating the move gives $(m-1,1)$.
The content sum is therefore largest among nontrivial partitions at
$(m-1,1)$, where it is $\binom{m-1}{2}-1$.  Hence the displayed scalar is at
least $m$, and the result follows.
\end{proof}

The following is the generalised Octopus inequality of Alon, Kozma, and
Puder for disjoint sets~\cite[Theorem~1.7]{AlonKozmaPuder}.

\begin{theorem}\label{thm:hypergraph-disjoint}
Let $t\geqslant1$, let $q\in[n]$, let
$A_1,\ldots,A_t\subseteq[n]\setminus\set{q}$ be
pairwise disjoint, and let $c_1,\ldots,c_t\geqslant0$.  Set
$C=\sum_{i=1}^tc_i\abs{A_i}$.  Then
\[
 C\sum_{j=1}^tc_j
   \bigl(\alpha_{A_j\cup\set{q}}-\alpha_{A_j}\bigr)
 \succeq
 \sum_{i=1}^tc_i^2\alpha_{A_i}
 +\sum_{1\leqslant i<j\leqslant t}c_ic_j
   \bigl(\alpha_{A_i\cup A_j}-\alpha_{A_i}-\alpha_{A_j}\bigr).
\]
\end{theorem}

\begin{proof}
Empty sets may be omitted.  Fix a unitary representation.  The
projections $\proj_{A_i}^{(\abs{A_i})}$ commute.  For each $i$, every set
indexing an $\alpha$-term in the difference of the two sides either contains
$A_i$ or is disjoint from it.  Hence $\proj_{A_i}^{(\abs{A_i})}$ also commutes
with the difference.  Decompose
the representation into their common eigenspaces.  On the summand indexed
by $Z\subseteq[t]$,
$\alpha_{A_i}$ acts as zero for $i\in Z$ and as $\abs{A_i}\one$ for
$i\notin Z$.  On this summand, the difference between the two sides is
\[
 C\sum_{i\in Z}c_i\alpha_{A_i\cup\set{q}}
 -\sum_{\substack{i<j\\i,j\in Z}}c_ic_j\alpha_{A_i\cup A_j}
 +\sum_{\substack{i<j\\i,j\notin Z}}c_ic_j\alpha_{A_i\cup A_j}.
\]
The last sum is positive semidefinite, so it is enough to prove
\[
 C\sum_{i\in Z}c_i\alpha_{A_i\cup\set{q}}
 \succeq
 \sum_{\substack{i<j\\i,j\in Z}}c_ic_j\alpha_{A_i\cup A_j}.
\]
There is nothing to prove if $Z=\varnothing$, so assume that $Z$ is nonempty.

On this summand, Lemma~\ref{lem:hypergraph-projections} gives
$\alpha_{A_i\cup\set{q}}=
\sum_{a\in A_i}\alpha_{\set{a,q}}$ for $i\in Z$.  Assign weight $c_i$
to every $a\in A_i$.  The first two terms in the preceding display have
the following decomposition on the chosen summand:
\[
\begin{aligned}
 &C\sum_{i\in Z}c_i\sum_{a\in A_i}\alpha_{\set{a,q}}
 -\sum_{\substack{i<j\\i,j\in Z}}c_ic_j\alpha_{A_i\cup A_j}\\
 &={}
 \left(C-\sum_{i\in Z}c_i\abs{A_i}\right)
 \sum_{i\in Z}c_i\sum_{a\in A_i}\alpha_{\set{a,q}}\\
 &\quad+
 \left[
 \left(\sum_{i\in Z}c_i\abs{A_i}\right)
 \sum_{i\in Z}c_i\sum_{a\in A_i}\alpha_{\set{a,q}}
 -\sum_{\substack{i<j\\i,j\in Z}}c_ic_j
  \sum_{\substack{a\in A_i\\b\in A_j}}\alpha_{\set{a,b}}
 \right]\\
 &\quad+
 \sum_{\substack{i<j\\i,j\in Z}}c_ic_j
 \left(
  \sum_{\substack{a\in A_i\\b\in A_j}}\alpha_{\set{a,b}}
  -\alpha_{A_i\cup A_j}
 \right).
\end{aligned}
\]
The first term on the right is positive semidefinite.  After restriction to
the chosen summand, the square-bracketed term is positive semidefinite by
\eqref{eq:hypergraph-critical-22}; the terms with both indices in a single
$A_i$ vanish because $\alpha_{\set{a,b}}=0$ on the range of
$\proj_{A_i}^{(\abs{A_i})}$.  In the last line, the same observation and
Lemma~\ref{lem:hypergraph-projections} give
\[
 \sum_{\substack{a\in A_i\\b\in A_j}}\alpha_{\set{a,b}}
 =\sum_{\set{a,b}\subseteq A_i\cup A_j}\alpha_{\set{a,b}}
 \succeq\alpha_{A_i\cup A_j}.
\]
Every term in the decomposition is therefore positive semidefinite, and the
result follows.
\end{proof}

We next give a Pl\"ucker-theoretic proof of their generalised Octopus
inequality for sets of co-size one~\cite[Theorem~1.9]{AlonKozmaPuder}.

\begin{theorem}\label{thm:hypergraph-deletions}
Let $t\geqslant1$, let $q\in[n]$, let
$A_0\subseteq[n]\setminus\set{q}$, and let
$A_1,\ldots,A_t$ be distinct subsets of $A_0$ such that
$\abs{A_0\setminus A_i}=1$.  Let $c_0,\ldots,c_t\geqslant0$, and set
$C=\sum_{i=0}^tc_i\abs{A_i}$.  Then
\[
 C\sum_{j=0}^tc_j
  \bigl(\alpha_{A_j\cup\set{q}}-\alpha_{A_j}\bigr)
 \succeq
 \sum_{i=0}^tc_i^2\alpha_{A_i}
 +\sum_{0\leqslant i<j\leqslant t}c_ic_j
  \bigl(\alpha_{A_0}+\alpha_{A_i\cap A_j}\bigr).
\]
\end{theorem}

\begin{proof}
Put $m=\abs{A_0}$.  The cases $m\leqslant1$ are immediate, so assume
$m\geqslant2$.  All the operators are supported on
$A_0\cup\set{q}$.

On the permutation module
$\C^{A_0\cup\set{q}}\cong S^{(m+1)}\oplus S^{(m,1)}$, one has, for every
$B\subseteq A_0\cup\set{q}$,
\[
 \alpha_B=\sum_{\set{a,b}\subseteq B}\alpha_{\set{a,b}}.
\]
For $a\in A_0$, give $a$ the weight
$x_a=\sum_{0\leqslant i\leqslant t,\,a\in A_i}c_i$.  Then
$C=\sum_{a\in A_0}x_a$.  The sets $A_0\setminus A_i$, for $0\leqslant i\leqslant t$,
are pairwise disjoint and have cardinality at most one.  Hence, for each
$i<j$, the two $c_ic_j$ terms in the expansion of $x_ax_b$ contribute once,
and contribute twice precisely when $a,b\in A_i\cap A_j$.  It follows that,
for distinct $a,b\in A_0$, the coefficient of $\alpha_{\set{a,b}}$ on the
right-hand side is
\[
 \sum_{\substack{0\leqslant i\leqslant t\\a,b\in A_i}}c_i^2
 +\sum_{0\leqslant i<j\leqslant t}c_ic_j
 +\sum_{\substack{0\leqslant i<j\leqslant t\\
                    a,b\in A_i\cap A_j}}c_ic_j
 =x_ax_b.
\]
The asserted inequality on the permutation representation is therefore
\[
 C\sum_{a\in A_0}x_a\alpha_{\set{a,q}}
 \succeq
 \sum_{\set{a,b}\subseteq A_0}x_ax_b\alpha_{\set{a,b}},
\]
which is precisely \eqref{eq:hypergraph-critical-22}.  This proves the
result on the trivial representation and $S^{(m,1)}$.

It remains to consider $S^\lambda$, where $\lambda\vdash m+1$ and
$\lambda_1\leqslant m-1$.  When $\abs B=m$, the branching rule shows that
$\proj_B^{(m)}$ can be nonzero only for $\lambda=(m+1)$ or $(m,1)$, which
have already been treated.  The assertion for $\abs B=m+1$ is immediate.
Hence $\proj_B^{(\abs B)}$ vanishes on $S^\lambda$ whenever
$\abs B\geqslant m$, and $\alpha_B=\abs B\one$ for every such $B$.
Rearranging the required inequality gives
\[
\begin{aligned}
 &C\left(c_0\alpha_{A_0\cup\set{q}}
       +\sum_{i=1}^tc_i\alpha_{A_i\cup\set{q}}\right)
 -\left(Cc_0+c_0^2+c_0\sum_{i=1}^tc_i
       +\sum_{1\leqslant i<j\leqslant t}c_ic_j\right)\alpha_{A_0}\\
 &\quad\succeq
 \sum_{i=1}^tc_i(C+c_i+c_0)\alpha_{A_i}
 +\sum_{1\leqslant i<j\leqslant t}c_ic_j\alpha_{A_i\cap A_j}.
\end{aligned}
\]
The left-hand side acts by the scalar
\[
 C\left((m+1)c_0+m\sum_{i=1}^tc_i\right)
 -m\left(Cc_0+c_0^2+c_0\sum_{i=1}^tc_i
       +\sum_{1\leqslant i<j\leqslant t}c_ic_j\right).
\]
Using $C=mc_0+(m-1)\sum_{i=1}^tc_i$ and
\[
 \left(\sum_{i=1}^tc_i\right)^2
 =\sum_{i=1}^tc_i^2
  +2\sum_{1\leqslant i<j\leqslant t}c_ic_j,
\]
this scalar is
\[
 (m-1)\sum_{i=1}^tc_i(C+c_i+c_0)
 +(m-2)\sum_{1\leqslant i<j\leqslant t}c_ic_j.
\]
Consequently, if $m\geqslant3$, the difference between the two sides of
the rearranged inequality is exactly
\[
 \sum_{i=1}^t(m-1)c_i(C+c_i+c_0)\proj_{A_i}^{(m-1)}
 +\sum_{1\leqslant i<j\leqslant t}
   (m-2)c_ic_j\proj_{A_i\cap A_j}^{(m-2)}\succeq0.
\]
When $m=2$, the same calculation gives only the first sum.  This
proves the inequality in
$\C[\Sn_{A_0\cup\set{q}}]$.  The standard inclusion
$\C[\Sn_{A_0\cup\set{q}}]\hookrightarrow\C[\Sn_n]$ preserves positivity,
so the result follows in $\C[\Sn_n]$.
\end{proof}

\begin{remark}
The generalised Octopus inequality of Alon, Kozma, and Puder for sets with
large intersection~\cite[Theorem~1.8]{AlonKozmaPuder} concerns sets
$A_i=A_0\cup\set{a_i}$.  The corresponding Young subgroup projections need
not commute, so the common eigenspace reduction used above is no longer
available.  It would be interesting to determine whether a suitable
Pl\"ucker relation, compressed to the corner
\[
 \proj_{A_0}^{(\abs{A_0})}\,
 \C[\Sn_{A_0\cup\set{q,a_1,\ldots,a_t}}]\,
 \proj_{A_0}^{(\abs{A_0})}
\]
yields their inequality.
\end{remark}


\begin{thebibliography}{99}

\bibitem{AldousCaputoDurrettHolroydJungPuha2021}
D.~Aldous, P.~Caputo, R.~Durrett, A.~E.~Holroyd,
P.~Jung, and A.~L.~Puha,
\emph{The life and mathematical legacy of Thomas M. Liggett},
Notices Amer. Math. Soc. \textbf{68} (2021), no.~1,
67--79.

\bibitem{AlonKozma}
G.~Alon and G.~Kozma,
\emph{Comparing with octopi},
Ann. Inst. Henri Poincar\'e Probab. Stat. \textbf{56} (2020), no.~4,
2672--2685.

\bibitem{AlonKozmaPuder}
G.~Alon, G.~Kozma, and D.~Puder,
\emph{On the Aldous--Caputo spectral gap conjecture for hypergraphs},
Math. Proc. Cambridge Philos. Soc. \textbf{179} (2025), 259--298.

\bibitem{Cesi}
F.~Cesi,
\emph{A few remarks on the Octopus inequality and Aldous' spectral gap
conjecture},
Comm. Algebra \textbf{44} (2016), no.~1, 279--302.

\bibitem{Chen}
J.~P. Chen,
\emph{The moving particle lemma for the exclusion process on a weighted
graph},
Electron. Commun. Probab. \textbf{22} (2017), paper no.~47, 13 pp.

\bibitem{CLR}
P.~Caputo, T.~M. Liggett, and T.~Richthammer,
\emph{Proof of Aldous' spectral gap conjecture},
J. Amer. Math. Soc. \textbf{23} (2010), no.~3, 831--851.

\bibitem{FultonHarris}
W.~Fulton and J.~Harris,
\emph{Representation Theory: A First Course},
Graduate Texts in Mathematics, vol.~129,
Springer-Verlag, New York, 1991.

\bibitem{Hora}
A.~Hora,
\emph{The Limit Shape Problem for Ensembles of Young Diagrams},
SpringerBriefs in Mathematical Physics, vol.~17,
Springer, Tokyo, 2016.

\bibitem{KarpPurbhoo}
S.~N. Karp and K.~Purbhoo,
\emph{Universal Pl\"ucker coordinates for the Wronski map and positivity
in real Schubert calculus},
J. Amer. Math. Soc., to appear;
\href{https://arxiv.org/abs/2309.04645}{arXiv:2309.04645v2}.

\bibitem{Sagan}
B.~E. Sagan,
\emph{The Symmetric Group: Representations, Combinatorial Algorithms,
and Symmetric Functions},
second ed., Graduate Texts in Mathematics, vol.~203,
Springer, New York, 2001.

\end{thebibliography}
\end{document}